\documentclass[11pt]{article}
\usepackage{color}
\usepackage{amsmath}
\catcode`\@=11 \@addtoreset{equation}{section}

\catcode`\@=12

\usepackage{amssymb}
\usepackage{amsfonts}
\usepackage{xcolor}
\usepackage{graphics}
\usepackage{epsfig,psfrag,graphicx}
\usepackage{amssymb}
\usepackage{eepic,epic}
\usepackage{dsfont}

\usepackage{enumerate}

\usepackage{epsfig} 
\usepackage{amsmath} 
\usepackage{amssymb} 
\usepackage{amsbsy} 

\newtheorem{definition}{Definition}[section]
\newtheorem{theorem}[definition]{Theorem}
\newtheorem{lemma}[definition]{Lemma}

\newtheorem{proposition}[definition]{Proposition}
\newtheorem{remark}[definition]{Remark}

\newcommand{\nc}{\newcommand}
\nc{\qed}{\mbox{}\nolinebreak\hfill \rule{2mm}{2mm}} 
\nc{\weak}{\rightharpoonup}
\nc{\weakstar}{\stackrel{\ast}{\rightharpoonup}} 
\nc{\proof}{{\bf Proof: }} 
\renewcommand{\div}{{{\mathrm{div}}_x}\,}

\nc{\modular}[1]{{\stackrel{ #1}{\longrightarrow\,}}}

\def\vec#1{\boldsymbol{#1}}

\newcommand{\ue}		{\vec{u}_\varepsilon}

\newcommand{\n}		{\vec{n}}

\newcommand{\dx}		{\,{\rm d}x}

\newcommand{\dt}		{\, {\rm d}t}
\newcommand{\dtau}		{\,{\rm d}\tau}
\newcommand{\dxdt}	{\, {\rm d}x{\rm d}t}

\newcommand{\mc}{\mathcal}

\newcommand{\veps}{\varepsilon}
\newcommand{\vtheta}{\vartheta}
\newcommand{\what}{\widehat}

\newcommand{\vphi}{\varphi}
\newcommand{\oline}{\overline}

\newcommand{\R}{\mathbb{R}}

\newcommand{\N}{\mathbb{N}}
\newcommand{\Z}{\mathbb{Z}}

\newcommand{\T}{\mathbb{T}^1}

\newcommand{\TT}{\mathbb{T}}

\renewcommand{\div}{{\rm div}\,}

\newcommand{\Supp}{{\rm supp}\,}

\allowdisplaybreaks

\def\d{\partial}
\def\div{{\rm div}\,}

\title{\LARGE \bf{Fast rotation and inviscid limits for the SQG equation\\
 with general ill-prepared initial data
}}
\author{ \textsl{Leonardo Kosloff}$\,^1\;$, $\;$\textsl{Gabriele Sbaiz}$\,^{2}\;$ \vspace{.2cm} \\
\footnotesize{$\,^1\;$ \textsc{Universidade Estadual de Campinas}, \textit{IMECC-Departamento de Matem\'atica,}} \\ {\footnotesize Rua S\'ergio Buarque de Holanda, CEP 13083-859, Campinas, SP, Brazil}
\vspace{0.1cm} \\
\footnotesize{$\,^2\;$ \textsc{Universit\`a degli Studi di Trieste},} \\
\footnotesize{\textit{Dipartimento di Scienze Economiche, Aziendali, Matematiche e Statistiche “Bruno de Finetti”},} \\
\footnotesize{Via Valerio 4/1, 34127 Trieste, Italy} \vspace{0.1cm} \\
\footnotesize{\ttfamily{kosloff@ime.unicamp.br}$\,,\quad$ \ttfamily{gabriele.sbaiz@phd.units.it}} \vspace{.1cm}
}

\date{\small \today}

\begin{document}
\maketitle

\abstract{In the present paper, we study the fast rotation and inviscid limits for the 2-D dissipative surface quasi-geostrophic equation with a dispersive forcing term $A \mc R_1\vtheta$, in the domain $\Omega =\T \times \R$.
In the case when we perform the fast rotation limit (keeping the viscosity fixed), in the context of general ill-prepared initial data, we prove that the limit dynamics is described by a linear equation. On the other hand, performing the combined fast rotation and inviscid limits, we show that the initial data $\oline \vartheta_0$ is transported along the motion. 
The proof of the convergence is based on an application of the Aubin-Lions lemma. 

}

\paragraph*{\small 2020 Mathematics Subject Classification:}{\footnotesize 35Q86 
(primary);
35B25, 
76U60, 
35B40, 
76M45 
(secondary).}

\paragraph*{\small Keywords:} {\footnotesize SQG equation; fast rotation limit; inviscid limit; singular perturbation problem; ill-prepared data
.}

\section{Introduction}

In this paper, we are interested in the description of the surface temperature $\vtheta$ on the ocean (see e.g. \cite{C-R} and \cite{Val}). We consider the two-dimensional surface quasi-geostrophic (SQG) equation with dissipation determined
by a fractional Laplacian and a dispersive forcing term in the domain $\Omega =\T \times \R$ given by:

\begin{equation} \label{eq_i:SQG}
\begin{cases}
	\partial_t \vtheta + \div(\vtheta \vec u)+ \nu \Lambda \vtheta+A\mc R_1 \vtheta=0\  \\[2ex]
\vec u=\mc R^\perp \vtheta:=(-\mc R_2 \vtheta,\mc R_1 \vtheta)  \\[2ex]
\vtheta_{|t=0}=\vtheta_0 \,,
\end{cases}
\end{equation}
where $\vtheta$ is a real-valued scalar function, $\vec u$ is the
divergence-free velocity field and $A$ represents the amplitude parameter for the dispersive forcing term. To define the fractional Laplacian operator $\Lambda := \sqrt{-\Delta}$ we can adapt the Fourier transform for functions $f(x_1, x_2) \in L^1(\Omega)$ with $(n_1, \xi_2) \in \mathbb{Z} \times \R$, so that 
\[
\hat{f}(n_1, \xi_2) = \int_{\mathbb{Z} \times \R} e^{-i(n_1 x_1 + \xi_2 x_2)} f(x_1, x_2)\, dx_1\,dx_2\,,
\]
and $\Lambda$ is then defined by:
\begin{equation*}
\widehat{\Lambda f}(n_1, \xi_2)=(n_1^2 +\xi_2^2)^{1/2}\, \widehat{f}(n_1, \xi_2):=|\tilde{\xi}|\widehat{f}(n_1, \xi_2)\, .
\end{equation*}
Moreover, we have $\nu > 0$ and $\mc R_i := \d_i \Lambda^{-1}$, for
$i = 1, 2$, are the usual Riesz transforms. 
The non-local dissipative term $\Lambda \vtheta$ comes from the Ekman pumping mechanism (we refer to \cite{H-P-G-S}, \cite{L} and \cite{Ped} for details); while the presence of an environmental horizontal gradient $\mc R_1 \vtheta = \d_1 \Lambda^{-1}\vtheta$ represents the advection of a large-scale buoyancy coming from the meridional variation of the Coriolis force. The parameter $A > 0$ is the analogue of the Rossby number which determines the typically large weight of the Coriolis term in geophysical fluid-dynamics. 
The mathematical analysis of \eqref{eq_i:SQG} started with the work \cite{Ki-Na} by Kiselev and Nazarov, where the existence of smooth solutions in the torus was addressed.

In the case without the presence of the dispersive forcing term, the non-dissipative SQG equations ($\nu = 0$) are the two-dimensional analogue of the 3D Euler equations in vorticity form (see for e.g. \cite{C-M-T}), while the SQG equations are analogous to the 3D Navier-Stokes system; due to this analogy the global regularity of the SQG equations has been intensively studied in recent decades (we refer, instead, to \cite{B-M-N_EMF}, \cite{B-M-N_AA} and \cite{B-M-N_IUM} for the pioneering studies for the Navier-Stokes equations). For an overview about these equations, we refer to \cite{C-V-T} and references therein. However, the presence of the dispersive forcing term makes equations \eqref{eq_i:SQG} analogous to the Navier-Stokes-Coriolis (NSC) system, which is a fundamental geophysical model dealing with large-scale phenomena (we refer to \cite{C-D-G-G} for a more detailed discussion). The main advantage of the NSC system is that in the limit of vanishing Rossby number, the fast rotating term ``produces'' a stabilization effect which ensures the global well-posedness of strong solutions with large initial data, unlike the case of the Navier-Stokes equations (see also \cite{C-M-X} in this respect). In particular, in \cite{C-D-G-G} this was proved by establishing Strichartz estimates, which show how the dispersive phenomena weaken the non-linearity and stabilize NSC towards a 2D Navier-Stokes type system. Following this analogy, a similar result for the supercritical dispersive SQG equation was shown in \cite{C-M-X}, where in \eqref{eq_i:SQG} the dissipative term is represented by $\nu \Lambda^{2\alpha} \vtheta$, $\alpha <1/2$. In this context, we mention also the work \cite{Kos-N-P} (in the context of well-prepared initial data), where the main tool employed is the relative energy inequality, in order to treat the inviscid incompressible limit of the NSC system with large rotation (see \cite{F-J-N} and \cite{C-N}). Furthermore, the method used in \cite{Kos-N-P} allows to obtain the inviscid limit with fixed or no dispersion, providing an alternative to previous results \cite{Ber} and \cite{Wu} for the unforced SQG equation.

Motivated by the previous discussion, our goal is to perform the fast rotation and inviscid limits in the more general framework of \textit{ill-prepared} initial data, where strong convergence of the initial data for the limiting and target systems, or \textit{a priori} structural conditions for the dispersive estimates, are \textit{not} required. In this direction, we remark that, in the domain $\Omega =\T \times \R$, the stability of the SQG system without dispersive forcing and with horizontal dissipation, was recently proved in \cite{PW}. Thus, this result shows that periodic boundary conditions allow for a decomposition into the mean flow and its oscillations, in which global existence, and indeed, stability, can be obtained in the Sobolev space $H^2(\Omega)$, without explicit recourse to dispersive effects.

In our work, first of all, we study the regime when the rotational effects are predominant in the dynamics, keeping fixed the viscous coefficient $\nu>0$, i.e. 
\begin{equation}\label{scaling_Rossby}
A=\frac{1}{\veps}\, ,
\end{equation} 
for a given $\veps \in ]0,1]$. Next, we analyse the combined fast rotation and inviscid limits where 
\begin{equation}\label{scaling_combined}
A=\frac{1}{\veps}\quad \quad \text{and}\quad \quad \nu=\nu(\veps)=\veps^{\alpha}\,,\quad \alpha > 0.
\end{equation}
The scaling for $\nu$ in \eqref{scaling_combined} is motivated by the physical Stommel boundary layer model for the western intensification of oceanic currents, where the Ekman pumping dissipation must be taken into account. In the inviscid limit case, this scaling would typically be written with $-1 < \alpha < -2/3$, but relabeling the parameter $\varepsilon$ leads to \eqref{scaling_combined}. Note also that we disregard the frictional sublayer needed to adjust the no-slip boundary conditions (for more details, we refer to \cite{Desj-Gre} and \cite{GV-Paul}).


In order to prove our results and get the improvement to the more general ill-prepared data, on the one hand, we take advantage of the special form of test functions, belonging to the kernel of the singular perturbation operator; on the other hand, we employ the structure of the system to find compactness properties for the means of temperature $\langle \vtheta_\veps \rangle$. The special form of the test functions combined 
with the \textit{Aubin-Lions lemma} will allow us to pass to the limit in the equations. Moreover, we need to work with periodic boundary conditions in the domain $\Omega$, so that non-trivial test functions are still allowed by the compact support condition (see remark \ref{periodic-need} after Theorem \ref{nolimitdispthm}).

The strategy employed is a standard matter in the context of singular perturbations problems and it consists in the following steps:
\begin{itemize}
\item[(i)] develop an existence theory, which holds for any value of $\veps>0$ fixed;
\item[(ii)] state uniform bounds for the family of solutions in order to extract weak limit points;
\item[(iii)] find the constraints that the limit points have to satisfy;
\item[(iv)] pass to the limit for test functions in the kernel of the singular perturbation operator.
\end{itemize}

To conclude this part, let us mention that, due to the stabilization effects of the Coriolis force, we expect convergence towards a linear equation in the case of scaling \eqref{scaling_Rossby}, while taking the combined scaling \eqref{scaling_combined} we will show that the limit dynamics ``collapses''. This will tell us that the initial limit profile is transported along the motion.

\medbreak
Let us now give an overview of the paper.
In Section \ref{s:result} we collect our assumptions and we state our main results. In Section \ref{s:sing-pert} we study the singular perturbation part of the equations,
recalling the uniform bounds on our family of weak solutions and establishing constraints that the limit points have to satisfy. Section \ref{s:proof} is devoted to the proof of the convergence
results for the fast rotation limit and the combined fast rotation and inviscid limits. 

\paragraph*{Some notation and conventions.} \label{ss:notations}

Let $B\subset\R^2$. 
The symbol $C_c^\infty (B)$ denotes the space of $\infty$-times continuously differentiable functions on $\R^2$ and having compact support in $B$. The dual space $\mc D^{\prime}(B)$ is the space of
distributions on $B$. 
Given $p\in[1,+\infty]$, by $L^p(B)$ we mean the classical space of Lebesgue measurable functions $g$, where $|g|^p$ is integrable over the set $B$ (with the usual modifications for the case $p=+\infty$).
We use also the notation $L_T^p(L^q)$ to indicate the space $L^p\big([0,T];L^q(B)\big)$, with $T>0$.
Given $k \geq 0$, we denote by $W^{k,p}(B)$ the Sobolev space of functions which belongs to $L^p(B)$ together with all their derivatives up to order $k$. When $p=2$, we alternately use the
notation $W^{k,2}(B)$ and  $H^k(B)$.
We denote by $\dot{W}^{k,p}(B)$ the corresponding homogeneous Sobolev spaces, i.e. $\dot{W}^{k,p}(B) = \{ g \in L^1_{\rm loc}(B)\, : \, D^\alpha g \in L^p(B),\ |\alpha| = k \}$.
Recall that $\dot{W}^{k,p}$ is the completion of $C^\infty_c(\overline{B})$ with respect to the $L^p$ norm of the $k$-th order derivatives.
For the sake of simplicity, we will omit from the notation the set $B$, that we will explicitly point out if needed.

We denote by $\T$ the one-dimensional flat torus $\T:=[-1,1] / \sim$, where $\sim$ denotes the equivalence relation which identifies $-1$ and $1$.

In the whole paper, the symbols $c$ and $C$ will denote generic multiplicative constants, which may change from line to line, and which do not depend on the small parameter $\veps$.
Sometimes, we will explicitly point out the quantities that these constants depend on, by putting them inside brackets.

Let $\big(f_\veps\big)_{0<\veps\leq1}$ be a sequence of functions in a normed space $X$. If this sequence is bounded in $X$,  we use the notation $\big(f_\veps\big)_{\veps} \subset X$.

\subsection*{Acknowledgements}
{\small 

L. Kosloff was supported by FAPESP - Brazil grant 2019/16537-0. G. Sbaiz is member of the INdAM (Italian Institute for Advanced Mathematics) group.

Moreover, the authors acknowledge Francesco Fanelli for his careful remarks that improve a lot the presentation of this work.
}

\section{Setting of the SQG problem and main statements} \label{s:result}

In this section, we formulate our working hypotheses (see Subsection \ref{ss:FormProb}) and we state our main results
(in Subsection \ref{ss:results}).

\subsection{Formulation of the problem} \label{ss:FormProb}

In this subsection, we present the rescaled SQG system with the dispersive forcing term, which we are going to consider in our study, and we
formulate the main working hypotheses. 
The material of this part is mostly classical: unless otherwise specified, we refer to \cite{Kos-N-P} for details.

 \subsubsection{Primitive system}\label{sss:primsys}
To begin with, let us introduce the ``primitive system'', i.e. the rescaled SQG system,
supplemented with the scaling \eqref{scaling_Rossby} presented in the introduction, where $\veps\in\,]0,1]$ is a small parameter.
Thus, the system consists of the momentum equation and the quasi-geostrophic balance: respectively,
\begin{align}
&\partial_t \vtheta_\veps + \div( \vtheta_\veps \ue)+ \nu \Lambda \vtheta_\veps+\frac{1}{\veps}\mc R_1 \vtheta_\veps=0 \label{eq_mom} \\
&\ue=\mc R^\perp \vtheta_\veps:=(-\mc R_2 \vtheta_\veps,\mc R_1 \vtheta_\veps) \, .\label{eq_Q-G_bal}
\end{align}
The unknown is the fluid surface temperature $\vtheta_\veps=\vtheta_\veps (t,x)$, with $t \in \R_+$ and $x\in \Omega$. 
\begin{remark}
In the case of scaling \eqref{scaling_combined}, we replace $\nu$ with $\veps^{\alpha},\, \alpha > 0$. 
\end{remark}


\subsubsection{Initial data and finite energy weak solutions} \label{sss:data-weak}

We address the singular perturbation problem described in subsection \ref{sss:primsys}, with scaling \eqref{scaling_Rossby}, for general \emph{ill prepared initial data}, in the framework of \emph{finite energy weak solutions} (see e.g. \cite{C-M-X}).
Since we work with weak solutions based on dissipation estimates, we need to assume that the initial data satisfy the following bound:
\begin{equation}\label{bound_theta_0}
\sup_{\veps\in \, ]0,1]}\|\vtheta_{0,\veps}\|_{L^2(\Omega)}\leq C\, .
\end{equation}

Thanks to the previous uniform estimate, up to extraction, we can argue that
\begin{equation}
\oline \vtheta_0:= \lim_{\veps \rightarrow 0}\vtheta_{0,\veps}\; ,
\end{equation}
where we agree that the previous limit is taken in the weak topology of $L^2(\Omega)$.

\medbreak


Let us specify better what we mean for \emph{finite energy weak solution} (see \cite{C-D-G-G} for details). 
\begin{definition}
We say that $\vtheta_\veps$ is a \textit{weak solution} to the dispersive SQG system \eqref{eq_mom}-\eqref{eq_Q-G_bal} in $[0,T[\, \times \Omega$ (for some time $T>0$) with the initial condition $\vtheta_{0,\veps}$, if:
\begin{itemize}
\item[(i)] $\vtheta_\veps \in L^\infty ([0,T[\, ;L^2(\Omega))\cap L^2([0,T[\, ;\dot{H}^{1/2}(\Omega))$;
\item[(ii)] the momentum equation is satisfied in a weak sense: for any $\varphi \in C_c^\infty ([0,T[\, \times \Omega)$, one has 
\begin{equation}\label{weak_mom}
-\int_0^T\int_{\Omega}\left(\vtheta_\veps \, \d_t\vphi + \vtheta_\veps \ue\cdot \nabla \vphi+\nu \Lambda^{1/2}\vtheta_\veps\,  \Lambda^{1/2}\vphi-\frac{1}{\veps}\mc R_1 \vtheta_\veps \, \vphi\right)\, \dx\dt=\int_{\Omega}\theta_{0,\veps}\, \vphi(0)\, \dx\, ;
\end{equation}
\item[(iii)] the quasi-geostrophic balance is satisfied in $\mc D^\prime (]0,T[\, \times \Omega)$;
\item[(iv)]the solution satisfies the following energy estimate:
\begin{equation}\label{energy_ineq}
\|\vtheta_\veps(T)\|^2_{L^2}+2\nu \int_0^T\|\Lambda^{1/2}\vtheta_\veps (\tau)\|^2_{L^2}\,\dtau \leq \|\vtheta_{0,\veps}\|^2_{L^2}\quad \text{ for all }\quad T>0\, .
\end{equation} 
\end{itemize} 
\end{definition}
The solution is \textit{global} if the previous conditions are satisfied for all $T>0$.	


\medskip

Under the previous assumptions (collected in subsections \ref{sss:primsys} and here above), at any \emph{fixed} value of the parameter $\veps\in\,]0,1]$,
the existence of a global in time finite energy weak solution $\vtheta_\veps$ to system SQG, related to the initial data
$\vtheta_{0,\veps}$, has been proved in e.g. \cite{Res} (see also \cite{C-M-X} in this respect). 
\begin{remark}
The term involving the Riesz transform $\mc R_1$ does not contribute to the energy estimate \eqref{energy_ineq}, because for any $s\in \R$, it holds:
\begin{equation}
\langle \Lambda^s \mc R_1 \theta\,,\; \Lambda^s \theta \rangle_{L^2}= \int_{\Omega}\Lambda^s \mc R_1 \theta\, \oline{\Lambda^s \theta} \, \dx =-\int_{\Omega}i\n_1|\tilde{\xi}|^{s-1}\hat\theta\, \oline{|\tilde{\xi}|^s\hat\theta}\, {\rm d}\tilde{\xi} =-\langle \Lambda^s \theta\,,\; \Lambda^s \mc R_1 \theta \rangle_{L^2} \, . 
\end{equation}
\end{remark}

\medskip

\subsection{Main results}\label{ss:results}

\medbreak
We can now state our main results. The first statement concerns the case when the rotational and viscosity effects, with fixed $\nu$, are predominant in the dynamics. 
\begin{theorem}\label{thm:fast_lim}
For any fixed value of $\veps \in \, ]0,1]$, assume the initial data $\vtheta_{0,\veps}$ verifies the hypothesis in subsection \ref{sss:data-weak} and let $\vtheta_\veps$ be a corresponding weak solution to system \eqref{eq_mom}-\eqref{eq_Q-G_bal}. Then, one has the following convergence property, for any $T>0$:
\begin{equation}
\vtheta_\veps \weakstar \oline \vtheta \qquad \qquad \mbox{ weakly-$*$ in }\qquad \qquad L^\infty_T(L^2(\Omega))\cap L^2_T(\dot H^{1/2}(\Omega))\, .
\end{equation} 
In addition, $\oline \vtheta$ is a weak solution to the following linear equation in $\R_+\times \Omega$:
\begin{equation}\label{limit_mom}
\d_t \oline \vtheta+\nu \Lambda_2\oline \vtheta=0\, ,
\end{equation}
supplemented with the initial condition $\oline \vtheta_{|t=0}=\oline \vtheta_0$.
\end{theorem} 

The previous theorem characterize the limit dynamics of system \eqref{eq_mom}-\eqref{eq_Q-G_bal} when one consider a fast rotation regime. In contrast with \cite{C-M-X}, in Theorem \ref{thm:fast_lim}, we are able to consider data which are \textit{ill-prepared}.

\medskip 

Instead, in the fast rotation and inviscid limits, we need more regularity on $\vtheta_\veps$ since we have to control $\Lambda \vtheta_\veps$. In this case, the limit dynamics is trivial and tell us that the limit temperature $\oline \vtheta$ is constant (in time).
\begin{theorem}\label{thm:inviscid dynamics}
For any fixed value of $\veps \in \, ]0,1]$, let the initial data $\vtheta_{0,\veps}\in H^s(\Omega)$ for $s>2$. Let $\vtheta_\veps$ be a corresponding solution to system \eqref{eq_mom}-\eqref{eq_Q-G_bal}. Then, one has the following convergence property, for $s>2$:
\begin{equation}
\vtheta_\veps \weakstar \oline \vtheta \qquad \qquad \mbox{ weakly-$*$ in }\qquad \qquad L^\infty_T(H^s(\Omega))\, .
\end{equation} 
Moreover, one deduces the relation (in the weak sense) 
$$\d_t \oline \vtheta=0\, ,$$ 
with the initial condition $\oline \vtheta_{|t=0}=\oline \vtheta_0$.
\end{theorem} 

\begin{remark}
We point out that the condition $s>2$ is necessary to have the embedding of Sobolev spaces $H^s$ in the space $W^{1,\infty}$ of globally Lipschitz functions (see the Appendix \ref{app:LP} for more details in this respect). 
\end{remark}

\section{Inspection of the singular perturbation}\label{s:sing-pert}
The purpose of this section is twofold. First of all, in Subsection \ref{ss:unif-est} we recall the uniform bounds and further properties for our family of weak solutions. Then, we study the singular operator underlying to the primitive SQG equations, and determine constraints that the limit points of our family of weak solutions have to satisfy (see Subsection \ref{ss:ctl1}).

\subsection{Uniform bounds}\label{ss:unif-est}

In this section we will state the uniform bounds on the sequence $\bigl(\vtheta_\veps)_\veps$. 

Indeed, with the energy estimate \eqref{energy_ineq} at hand, we can derive uniform bounds for our family of weak solutions. To begin with, we point out that, owing to the assumption \eqref{bound_theta_0}, the right-hand side of \eqref{energy_ineq} is \textit{uniformly bounded} for all $\veps \in\, ]0,1]$. Then, one has
\begin{align}
\sup_{T\in \R_+}\|\vtheta_\veps (T)\|_{L^2(\Omega)}\leq c\label{unif_bound_theta}\\
\int_0^T\|\Lambda^{1/2}\vtheta_\veps (\tau)\|^2_{L^2(\Omega)}\, \dtau\leq c \quad \text{for all }T>0\, .\label{unif_bound_lam_theta}
\end{align}
Due to the relation \eqref{eq_Q-G_bal} and the fact that the Riesz operator is a $0$-th order operator, one gets also
\begin{equation}
\sup_{T\in \R_+}\|\ue (T)\|_{L^2}+\int_0^T\|\Lambda^{1/2}\ue (\tau)\|^2_{L^2}\, \dtau\leq c\, ,
\end{equation}
where the generic constant $c$ is independent of $\veps> 0$.

In view of the previous properties, there exist $\oline \vtheta$, $\oline{\vec u} \in L^\infty_T(L^2)\cap L^2_T(\dot H^{1/2})$ such that (up to the extraction of a suitable subsequence) one has
\begin{equation}
\vtheta_\veps \weakstar \oline \vtheta \qquad \qquad \text{ and }\qquad \qquad \ue \weakstar \oline{\vec u} \, .
\end{equation}

\subsection{Constraints on the limit dynamics}\label{ss:ctl1}

In this section, we establish some properties that the limit points of the family $\bigl(\vtheta_\veps\bigr)_\veps$ have to satisfy.
These are static relations, which do not characterise the limit dynamics yet.

\begin{proposition} \label{nolimitdispthm}
Let $(\vtheta_\veps)_\veps$ be a family of weak solutions, related to initial data $(\vtheta_{0,\veps})_\veps$ verifying hypothesis of subsection \ref{sss:data-weak}. Let $\oline \vtheta$ be a limit point of the previous sequence as identified in Subsection \ref{ss:unif-est}. Then, one deduces the relation 
\begin{equation}\label{constraint-limit}
\d_1 \oline \vtheta=0 \qquad \qquad \text{a.e. in }\; \R_+\times \Omega\, . 
\end{equation} 
\end{proposition}

\begin{proof}
Let us consider the weak formulation of \eqref{eq_mom}. We test it against $\veps \vphi$ where $\vphi$ is a compactly supported test function in $C_c^\infty (\R_+\times \Omega)$. Denoting $[0,T]\times K=\Supp \vphi$ with $\vphi(T,\cdot)=0$ (and $T>0$), we have
\begin{equation}\label{eq:mom_constraint}
-\veps\int_0^T\int_K \left(\vtheta_\veps \d_t \vphi+\vtheta_\veps \ue \cdot \nabla \vphi+\nu \Lambda^{1/2}\vtheta_\veps \, \Lambda^{1/2}\vphi \right)\, \dxdt +\int_0^T\int_K\mc R_1 \vtheta_\veps \, \vphi \,  \dxdt=\veps \int_K\vtheta_{0,\veps}\vphi(0) \, \dx.
\end{equation}
By uniform bounds \eqref{unif_bound_theta} and \eqref{unif_bound_lam_theta}, the first three integrals on the left-hand side converge to 0, when tested against any smooth compactly supported $\vphi$. 

Analogously, thanks to the bound \eqref{bound_theta_0} on the initial data, also the term on the right-hand side of \eqref{eq:mom_constraint} vanishes. 

Then, passing to the limit for $\veps\rightarrow 0$, we find
\begin{equation*}
\int_0^T\int_K\mc R_1 \oline \vtheta \, \vphi \,  \dxdt=0\, ,
\end{equation*}   
for any test function $\vphi \in C_c^\infty (\R_+\times \Omega)$, which in particular implies
\begin{equation*}
\mc R_1 \oline \vtheta=0 \qquad \qquad \text{ a.e. in }\; \R_+\times \Omega\, .
\end{equation*}
At this point, employing the Fourier transform, one can deduce that 
\begin{equation}
\d_1 \oline \vtheta=0 \qquad \qquad \text{ a.e. in }\; \R_+\times \Omega\, .
\end{equation}
This completes the proof of the proposition. 
\qed
\end{proof}

\begin{remark}\label{periodic-need}
Observe that the proof above relies on the fact that the compact support of test functions does not force them to be 0 everywhere, unlike the case of the domain $\R^2$, where the condition that the test functions are only dependent on the vertical variable implies they must vanish inside the support as well.
\end{remark}

\section{Limit dynamics}\label{s:proof}

\subsection{Convergence to the fast rotation limit dynamics}\label{ss:fast rotation limit}
In this section, we will show the convergence of momentum equation \eqref{weak_mom} towards the linear equation \eqref{limit_mom} depicted in Theorem \ref{thm:fast_lim}.  

\medbreak
The uniform bounds of Subsection \ref{ss:unif-est} are
not enough for proving convergence in the weak formulation of the momentum equation: the main problem relies on identifying the weak limit of the convective term
$\div (\vtheta_\veps \ue)$. 


At this point, we rewrite the convective term in the weak formulation 
\begin{equation}\label{rel_average}
-\int_0^T\int_\Omega \vtheta_\veps \ue \cdot \nabla \psi \, \dx\dt= -\int_0^T\int_{\T \times \R}\vtheta_\veps (u_\veps)_2\,  \d_2 \psi \, \dx \dt =-2\int_{\R}\langle \vtheta_\veps (u_\veps)_2\rangle\,  \d_2 \psi\, \dx \dt\, ,
\end{equation}
where the test-function $\psi$ is defined as
\begin{equation}\label{test-func}
\psi \in C^\infty_c([0,T[\, \times \Omega;\R) \quad \text{ such that }\quad \d_1\psi=0\, .
\end{equation}
The previous relation \eqref{rel_average} pushes our attention to find compactness properties for $\langle \vtheta_\veps \rangle$.

Taking the mean in equation \eqref{eq_mom} and thanks to the fact that the mean of the Riesz term $\langle \mc R_1 \vtheta_\veps\rangle$ is zero, we get 
\begin{equation}\label{mean equation}
\d_t \langle \vtheta_\veps \rangle=-\div \langle \vtheta_\veps \ue \rangle - \nu \Lambda \langle \vtheta_\veps \rangle\, .
\end{equation}
To analyse the former term in the right-hand side of \eqref{mean equation}, we employ the paradifferential calculus (see Chapter 2 of \cite{B-C-D}). Using the Bony decomposition, one can find that the term $\vtheta_\veps \ue$ is in $L^2_T B^{-1/2}_{2,1}$ and thanks to the embedding $B^{-1/2}_{2,1}\hookrightarrow H^{-1/2}$, we argue that $\div \langle \vtheta_\veps \ue \rangle$ belongs to $L^2_T H^{-3/2}$. 

Then, recalling equation \eqref{mean equation}, we deduce that $(\d_t \langle \vtheta_\veps \rangle)_\veps$ is uniformly bounded in $L^2_T H^{-3/2}_{\rm loc}$, which implies $(\langle \vtheta_\veps \rangle )_\veps \subset W^{1,2}_T H^{-3/2}_{\rm loc}$. Moreover, we already know that $(\langle \vtheta_\veps \rangle)_\veps \subset L^\infty_T L^2\cap L^2_T \dot H^{1/2}$. Therefore, the Aubin-Lions lemma gives compactness of $(\langle \vtheta_\veps \rangle)_\veps$ in e.g. $L^2_T L^2_{\rm loc}$. Thus, we deduce the strong convergence (up to extraction) for $\veps \rightarrow 0$: 
$$ \langle \vtheta_\veps \rangle \rightarrow \langle \oline \vtheta \rangle \quad \text{ in }\quad L^2_T L^2_{\rm loc}\, . $$

This implies that 
\begin{equation}\label{eq:lim_conv}
\langle \vtheta_\veps \ue \rangle \rightarrow \langle\oline \vtheta\,  \oline{\vec u}\rangle \quad \quad \text{in}\quad \quad \mathcal{D}^\prime (\R_+ \times \Omega )\, .
\end{equation}

\subsubsection{Description of the limit system}
With the convergence established in \eqref{eq:lim_conv}, we can pass to the limit in the equations.

To begin with, we take a test-function $\psi$ (for $T>0$) as in \eqref{test-func}.
For such a $\psi$, the Riesz term $\mc R_1 \vtheta_\veps$ vanishes identically. Hence, the momentum equation in its weak formulation reads: 
\begin{equation}\label{weak_mom-1}
-\int_0^T\int_{\Omega}\left(\vtheta_\veps \, \d_t\psi + \vtheta_\veps \ue\cdot \nabla \psi+\nu \Lambda^{1/2}\vtheta_\veps\,  \Lambda^{1/2}\psi\right)\, \dx\dt=\int_{\Omega}\theta_{0,\veps}\, \psi(0)\, \dx\, .
\end{equation}
Making use of the uniform bounds of Subsection \ref{ss:unif-est}, we can pass to the limit in the $\d_t$ term and in the viscosity term. Moreover, our assumptions imply that $\vtheta_{0,\veps}\rightharpoonup \oline \vtheta_0$ in e. g. $L^2_{\rm loc}$. Thanks to the convergence identified in \eqref{eq:lim_conv} for the convective term, letting $\veps\rightarrow 0$, we may infer that
  \begin{equation}\label{lim_weak_mom}
-\int_0^T\int_{\Omega}\left(\oline \vtheta\, \d_t\psi +\nu \Lambda^{1/2}\oline \vtheta\,  \Lambda^{1/2}\psi\right)\, \dx\dt=\int_{\Omega}\oline \theta_{0}\, \psi(0)\, \dx\, ,
\end{equation}
where we have also employed the constraint \eqref{constraint-limit} and the special structure of the test functions \eqref{test-func}. This concludes the proof of Theorem \ref{thm:fast_lim}.
\subsection{Proof of the convergence in the inviscid and fast rotation case}\label{s:fast-inviscid_limit}

In this subsection, we consider the case when $\nu=\nu(\veps)=\veps^{\alpha},\, \alpha > 0$. In that scaling, the energy estimate reads
\begin{equation}\label{energy_ineq-1}
\|\vtheta_\veps(T)\|^2_{L^2}+2\veps^{\alpha} \int_0^T\|\Lambda^{1/2}\vtheta_\veps (\tau)\|^2_{L^2}\,\dtau \leq \|\vtheta_{0,\veps}\|^2_{L^2}\quad \text{ for all }\quad T>0\, .
\end{equation}
Therefore, one completely loses the uniform control on $\Lambda^{1/2}\vtheta_\veps$. In order to recover a similar bound to \eqref{unif_bound_lam_theta}, we have to consider strong solutions with the following regularity properties:
\begin{equation*}
\vtheta_\veps \in L^\infty ([0,T[\, ;H^s(\Omega))\quad \text{and}\quad s>2\, .
\end{equation*}

The aim now is to show higher order estimates for the temperature. We have already presented in \eqref{energy_ineq-1} the $L^2$ estimate. In order to obtain the $H^s$ control, we employ the Littlewood-Paley decomposition (see the Appendix \ref{app:LP}). Applying the dyadic blocks $\Delta_j$ to equation \eqref{eq_mom} (see e.g. \cite{S}), we obtain
\begin{equation*}
\d_t \Delta_j \vtheta_\veps +\ue \cdot \nabla \Delta_j \vtheta_\veps +\veps^\alpha \Lambda\Delta_j \vtheta_\veps+\frac{1}{\veps}\mc R_1
\Delta_j \vtheta_\veps=[\ue \cdot \nabla, \Delta_j ]\vtheta_\veps\, .
\end{equation*}

Multiplying by $\Delta_j \vtheta_\veps$ and using the orthogonality of the Riesz term, it follows that 
\begin{equation*}
\d_t \| \Delta_j \vtheta_\veps \|_{L^2}^2+2\veps^\alpha \|\Lambda^{1/2}\Delta_j \vtheta_\veps \|_{L^2}^2\leq 2\|[\ue \cdot \nabla , \Delta_j]\vtheta_\veps \|_{L^2}\|\Delta_j \vtheta_\veps \|_{L^2}\, .
\end{equation*}
In particular, we get 
\begin{equation*}
\|\Delta_j \vtheta_\veps \|_{L^2}\leq \|\Delta_j \vtheta_{0,\veps}\|_{L^2}+C \int_0^t\|[\ue \cdot \nabla , \Delta_j]\vtheta_\veps \|_{L^2} \dtau \, .
\end{equation*}
At this point, thanks to the commutator estimates (we refer e.g. to \cite{S}) we have 
$$ \|[\ue \cdot \nabla, \Delta_j]\vtheta_\veps\|_{L^2}\leq C c_j(t)\left(\|\vtheta_\veps\|_{L^\infty}\|\ue \|_{H^s}+\|\vtheta_\veps\|_{H^s}\|\ue \|_{L^\infty}\right)\leq C c_j(t)\|\vtheta_\veps\|_{L^2}^2 \, , $$
where $(c_j(t))_{j\geq -1}$ is a sequence in the unit ball of $\ell^2$. 

After summing on indices $j\geq -1$, we finally derive for all $t\geq 0$:
\begin{equation}\label{eq_energy_H^s}
\|\vtheta_\veps (t)\|_{H^s}\leq \|\vtheta_{0,\veps}\|_{H^s}+C\int_0^t \|\vtheta_\veps (\tau )\|_{H^s}^2
\dtau \, .
\end{equation}
The scope now is finding a time $T^\ast>0$ for which the solutions are uniformly bounded (in $\veps$) in the interval $[0,T^\ast]$.

We define $T^\ast_\veps >0$ such that 
\begin{equation}\label{time_veps}
T^\ast_\veps:=\sup \left\{t>0: \int_0^t \|\vtheta_\veps\|_{H^s}^2\leq \|\vtheta_{0,\veps}\|_{H^s}\right\}\, .
\end{equation}
Then, we deduce $\|\vtheta_\veps (t)\|_{H^s}\leq C\|\vtheta_{0,\veps}\|_{H^s}$ for all times $t\in [0, T_\veps^\ast]$ and for some positive constant $C=C(s)$. Therefore, for all $t\in [0,T^\ast_\veps]$ we gather
$$ \int_0^t \|\vtheta_\veps (\tau )\|_{H^s}^2 \dtau \leq Ct\|\vtheta_{0,\veps}\|_{H^s}^2 \, .$$
By using the definition \eqref{time_veps} of $T^\ast_\veps$, we finally argue that
$$ T_\veps^\ast\geq \frac{C}{\|\vtheta_{0,\veps}\|_{H^s}}\, . $$
In particular, this implies that there exists a time $T^\ast >0$ such that
$$ T^\ast:= \sup_{\veps >0}T^\ast_\veps >0\, . $$

From the definition of $T^\ast >0$ and relation \eqref{eq_energy_H^s}, one can obtain that 
\begin{equation}
\sup_{\veps \in ]0,1]}\|\vtheta_\veps \|_{L^\infty_{T^{\ast}} (H^s)}\leq C\, ,
\end{equation}
and this allows to control also the dissipative term in the relation \eqref{energy_ineq-1}.

At this point, since those strong solutions are in particular weak solutions, one can adapt the arguments developed in subsections \ref{ss:unif-est} and \ref{ss:fast rotation limit} to conclude. Indeed, letting $\veps\rightarrow 0$, one has  
\begin{equation}\label{lim_weak_mom-1}
-\int_0^{T^{\ast}}\int_{\Omega}\oline \vtheta\, \d_t\psi\, \dx\dt=\int_{\Omega}\oline \theta_{0}\, \psi(0)\, \dx\, ,
\end{equation}
for all test functions $\psi$ defined as in \eqref{test-func}. This proves the limit dynamics in Theorem \ref{thm:inviscid dynamics}.


\newpage
\appendix

\section{Appendix -- Littlewood-Paley theory} \label{app:LP}

In this appendix, we present some tools from Littlewood-Paley theory, which we have exploited in our analysis.
We refer e.g. to Chapter 2 of \cite{B-C-D} for details.
For simplicity of exposition, we deal with the $\R^d$ case, with $d\geq1$; however, the whole construction can be adapted also to the $d$-dimensional torus $\TT^d$, and to the ``hybrid'' case
$\R^{d_1}\times\TT^{d_2}$.

First of all, we introduce the \emph{Littlewood-Paley decomposition}. For this,
we fix a smooth radial function $\chi$ such that $\Supp\chi\subset B(0,2)$, $\chi\equiv 1$ in a neighborhood of $B(0,1)$
and the map $r\mapsto\chi(r\,e)$ is non-increasing over $\R_+$ for all unitary vectors $e\in\R^d$.
Set $\varphi\left(\xi\right)=\chi\left(\xi\right)-\chi\left(2\xi\right)$ and $\vphi_j(\xi):=\vphi(2^{-j}\xi)$ for all $j\geq0$.
The dyadic blocks $(\Delta_j)_{j\in\Z}$ are defined by\footnote{We agree  that  $f(D)$ stands for 
the pseudo-differential operator $u\mapsto\mc{F}^{-1}[f(\xi)\,\what u(\xi)]$.} 
$$
\Delta_j\,:=\,0\quad\mbox{ if }\; j\leq-2,\qquad\Delta_{-1}\,:=\,\chi(D)\qquad\mbox{ and }\qquad
\Delta_j\,:=\,\varphi(2^{-j}D)\quad \mbox{ if }\;  j\geq0\,.
$$
For any $j\geq0$ fixed, we  also introduce the \emph{low frequency cut-off operator}
\begin{equation} \label{eq:S_j}
S_j\,:=\,\chi(2^{-j}D)\,=\,\sum_{k\leq j-1}\Delta_{k}\,.
\end{equation}
Note that $S_j$ is a convolution operator. More precisely, after defining
$$
K_0\,:=\,\mc F^{-1}\chi\qquad\qquad\mbox{ and }\qquad\qquad K_j(x)\,:=\,\mathcal{F}^{-1}[\chi (2^{-j}\cdot)] (x) = 2^{jd}K_0(2^j x)\,,
$$
for all $j\in\N$ and all tempered distributions $u\in\mc S'$, we have that $S_ju\,=\,K_j\,*\,u$.
Thus, the $L^1$ norm of $K_j$ is independent of $j\geq0$, hence $S_j$ maps continuously $L^p$ into itself, for any $1 \leq p \leq +\infty$.

The following property holds true: for any $u\in\mc{S}'$, then one has the equality $u=\sum_{j}\Delta_ju$ in the sense of $\mc{S}'$.
Let us also recall the so-called \emph{Bernstein inequalities}.
  \begin{lemma} \label{l:bern}
Let  $0<r<R$.   A constant $C$ exists so that, for any non-negative integer $k$, any couple $(p,q)$ 
in $[1,+\infty]^2$, with  $p\leq q$,  and any function $u\in L^p$,  we  have, for all $\lambda>0$,
$$
\displaylines{
{\Supp}\, \widehat u \subset   B(0,\lambda R)\quad
\Longrightarrow\quad
\|\nabla^k u\|_{L^q}\, \leq\,
 C^{k+1}\,\lambda^{k+d\left(\frac{1}{p}-\frac{1}{q}\right)}\,\|u\|_{L^p}\;;\cr
{\Supp}\, \widehat u \subset \{\xi\in\R^d\,:\, \lambda r\leq|\xi|\leq \lambda R\}
\quad\Longrightarrow\quad C^{-k-1}\,\lambda^k\|u\|_{L^p}\,
\leq\,
\|\nabla^k u\|_{L^p}\,
\leq\,
C^{k+1} \, \lambda^k\|u\|_{L^p}\,.
}$$
\end{lemma}   

By use of Littlewood-Paley decomposition, we can define the class of Besov spaces.
\begin{definition} \label{d:B}
  Let $s\in\R$ and $1\leq p,r\leq+\infty$. The \emph{non-homogeneous Besov space}
$B^{s}_{p,r}$ is defined as the subset of tempered distributions $u$ for which
$$
\|u\|_{B^{s}_{p,r}}\,:=\,
\left\|\left(2^{js}\,\|\Delta_ju\|_{L^p}\right)_{j\geq -1}\right\|_{\ell^r}\,<\,+\infty\,.
$$
\end{definition}
Besov spaces are interpolation spaces between Sobolev spaces. In fact, for any $k\in\N$ and~$p\in[1,+\infty]$
we have the chain of continuous embeddings $ B^k_{p,1}\hookrightarrow W^{k,p}\hookrightarrow B^k_{p,\infty}$,
which, when $1<p<+\infty$, can be refined to
$B^k_{p, \min (p, 2)}\hookrightarrow W^{k,p}\hookrightarrow B^k_{p, \max(p, 2)}$.
In particular, for all $s\in\R$ we deduce that $B^s_{2,2}\equiv H^s$, with equivalence of norms:
\begin{equation} \label{eq:LP-Sob}
\|f\|_{H^s}\,\sim\,\left(\sum_{j\geq-1}2^{2 j s}\,\|\Delta_jf\|^2_{L^2}\right)^{\!\!1/2}\,.
\end{equation}

As an immediate consequence of the first Bernstein inequality, one gets the following embedding result, which generalises Sobolev embeddings.
\begin{proposition}\label{p:embed}
The space $B^{s_1}_{p_1,r_1}$ is continuously embedded in the space $B^{s_2}_{p_2,r_2}$ for all indices satisfying $p_1\,\leq\,p_2$ and either
$s_2\,<\,s_1-d\big(1/p_1-1/p_2\big)$, or $s_2\,=\,s_1-d\big(1/p_1-1/p_2\big)$ and $r_1\leq r_2$.
\end{proposition}
In particular, we get the following chain of continuous embeddings:
$$ B^s_{p,r}\hookrightarrow W^{1,\infty} \, , $$
whenever the triplet $(s,p,r)\in \R\times [1,+\infty]^2$ satisfies 
\begin{equation}\label{cond_algebra}
s>1+\frac{d}{p} \quad \quad \quad \text{or}\quad \quad \quad s=1+\frac{d}{p} \quad \text{and}\quad r=1\, .
\end{equation}


\end{document}